\documentclass{amsart}
\usepackage{amsmath,amsfonts,amsthm,amssymb,indentfirst,epic,xurl,graphics,needspace}
\newtheorem{theorem}{Theorem}
\newtheorem{lemma}[theorem]{Lemma}
\newtheorem{corollary}[theorem]{Corollary}
\newtheorem{proposition}[theorem]{Proposition}
\newtheorem{definition}[theorem]{Definition}
\newtheorem{question}[theorem]{Question}

\numberwithin{equation}{section}
\numberwithin{theorem}{section}

\renewcommand{\r}{\mathrm}
\newcommand{\abs}{\r{abs}}

\mathchardef\mhyphen="2D
\newcommand{\2}{\mathbf{2}}
\usepackage[T1]{fontenc} % to get \k{e} in J\k{e}drzej

\begin{document}

\begin{center}
\texttt{Comments, corrections,
and related references welcomed, as always!}\\[.5em]
{\TeX}ed \today
\\[.5em]
\vspace{2em}
\end{center}

\title%
[Dimensions of products of posets]%
{Some frustrating questions on \\
dimensions of products of posets}
\thanks{%
Archived at \url{http://arxiv.org/abs/2312.12615}\,.
Readable at \url{http://math.berkeley.edu/~gbergman/papers/}.
The latter version may be revised more frequently than the former.
After publication, any updates, errata, related references,
etc., found will be noted at
\url{http://math.berkeley.edu/~gbergman/papers/abstracts}\,.\\
\hspace*{1.5em}Data sharing not applicable, as no datasets were
generated or analyzed in this work.
}


\author{George M.\ Bergman}
\address{Department of Mathematics\\
University of California\\
Berkeley, CA 94720-3840, USA}
\email{gbergman@math.berkeley.edu}

\begin{abstract}
The definition of the dimension of a poset is recalled.
For a subposet $P$ of a direct product of $d>0$ chains,
and an integer $n>0,$ a condition is developed which
implies that for any
family of $n$ chains $(T_j)_{j\in n},$ one has
\mbox{$\dim(P\times\prod_{j\in n}T_j)\leq d.$}
Applications are noted.

Open questions, old and new, on dimensions of product posets
are stated, and some other numerical invariants of posets that seem
useful for studying these questions are developed.
Some variants of the concept of the dimension of
a poset from the literature are also recalled.

In a final section, independent of the other results, it is
noted that by the compactness theorem of first-order logic,
an infinite poset $P$ has finite dimension $d$ if and only if $d$
is the supremum of the dimensions of its finite subposets.
\end{abstract}
\maketitle
% - - - - - - - - - - - - - - - - - - - - - - - - - - - - - -

\section{Definitions and examples}\label{S.defs}

We assume familiarity with the definition of a
{\em partially ordered set}, for which we will use the
short term {\em poset}, and of the special case
of a {\em totally ordered set}, also called a {\em chain}.
We will formally consider a poset $P$ to be an ordered
pair $(|P|,\,\leq_P),$ where $|P|$ is the underlying set
and $\leq_P$ the order relation; but when there is no danger of
ambiguity, we shall write $p\in P$ to mean $p\in |P|,$
and $\leq$ for~$\leq_P.$
We will follow standard notational conventions such as
letting $x\geq y$ denote $y\leq x,$
$<$ the conjunction of $\leq$ with $\neq,$
and $\nleq$ the negation of $\leq.$

Posets will always be understood to be nonempty.

Recall that a set-map $f:P\to Q$ between posets is called
{\em isotone} if
\begin{equation}\begin{minipage}[c]{35pc}\label{d.isotone}
$p\,\leq\,p' \implies f(p)\leq f(p')$\ \ for all $p,\,p'\in P,$
\end{minipage}\end{equation}
and an {\em embedding} if it is isotone and also satisfies
\begin{equation}\begin{minipage}[c]{35pc}\label{d.embdg}
$p\,\nleq\,p' \implies f(p)\nleq f(p')$\ \ for all $p,\,p'\in P.$
\end{minipage}\end{equation}
Clearly, an embedding of posets is one-to-one; but a
one-to-one isotone map need not be an embedding.

If $\leq$ is a partial ordering on a set $X,$ we will call
a partial ordering $\leq'$ on $X$ a {\em strengthening} of
$\leq$ if $\{(x,\,y)\,|\,x\leq y\}\subseteq
\{(x,\,y)\,|\,x\leq' y\};$ i.e., if the identity map of
$X$ is an isotone map \mbox{$(X,\,\leq)\to (X,\,\leq').$}

A {\em linearization} of a partial ordering $\leq$ on a
set $X$ means a {\em total} ordering of $X$ that
is a strengthening of $\leq.$
It is easy to verify that every partial ordering on a
set admits a linearization, and in fact that
\begin{equation}\begin{minipage}[c]{35pc}\label{d.lnzn}
Every partial ordering $\leq$ on a set $X$ is (as a set of
ordered pairs) the intersection of its \mbox{linearizations}
(Szpilrajn's Theorem \cite{Sz}).
\end{minipage}\end{equation}
(Idea of proof:  Given $x,y\in X$ such that $x\nleq y,$ show
that there is a strengthening $\leq'$ of $\leq$ such that $x>' y.$
Iterating this process, transfinitely if necessary, we get
a linearization of $\leq.$
Moreover, by our choice of the pair with $x\nleq y$ that we start with,
we can insure that any prechosen order-relation between a pair
of elements that does not
hold under $\leq$ fails to hold in some linearization.
Hence we have~\eqref{d.lnzn}.)

By the {\em product} $P\times Q$ of two posets $P$ and $Q$ one
understands the poset whose underlying set is
$|P|\times|Q|,$ ordered so that $(p,q)\leq (p',q')$
if and only if $p\leq p'$ and $q\leq q';$
and analogously for products of larger families of posets.
(These are in fact products in the category of posets and isotone
maps, though we shall not use category-theoretic language here.)
Thus~\eqref{d.lnzn} says
that every poset $P$ embeds in the product of the chains
obtained from $P$ using linearizations of $\leq_P.$
Occasionally, I shall refer to a product of posets as their
`direct product', when this seems desirable for clarity.

We now come to the concept we will be studying in this note,
due to B.\ Dushnik and E.\ W.\ Miller~\cite{D+M}:

\begin{definition}\label{D.dim}
The {\em dimension} $\dim(P)$ of a poset $P$ is the least cardinal
$\kappa$ such that

\textup{(i)}\ \ the relation $\leq_P$ is an
intersection of $\kappa$ total orderings on $|P|,$\\
equivalently,

\textup{(ii)}\ \ $P$ is embeddable in a product of $\kappa$ totally
ordered sets.
\end{definition}

In the above  definition, the implication (i)$\implies$(ii) is clear.
To get (ii)$\implies$(i) $(=$\cite[Theorem~10.4.2]{Ore})
first consider any isotone map
$f:P\to T$ where $T$ is a totally ordered set.
For each $t\in T,$ regard the inverse image of $t$ in $P$
as a subposet, and choose a linearization $\leq_t$ of
the partial ordering of that subposet of $P.$
We can now strengthen $\leq_P$ to a linear ordering $\leq'$
of $|P|$ by making $p\leq' q$ if either
$f(p) <_T f(q),$ or $f(p) = f(q)$ and $p\leq_{f(p)} q.$
Thus, given an embedding of $P$ in a product of $\kappa$
posets $T_\alpha,$ if we construct
from each of the projections $f_\alpha: P\to T_\alpha$ a linearization
$\leq'_\alpha$ of $\leq_P$ as above, we see that the intersection of
these linearizations will again be the partial ordering~$\leq_P.$

We understand the product of the empty family of sets to
be a singleton.
Thus a poset has dimension~$0$ if
and only if its underlying set is a singleton.

In the literature on dimensions of posets, condition~(i)
above is generally
the preferred definition; but here we will more often use~(ii).

This note will focus almost entirely on {\em finite-dimensional}
posets, though we will allow underlying sets of these to be infinite.
In indexing finite families of maps etc.,
we will follow the set-theorists' convention
\begin{equation}\begin{minipage}[c]{35pc}\label{d.n=}
For $n$ a natural number, $n\ =\ \{0,\dots,n{-}1\}.$
\end{minipage}\end{equation}

For positive integers $n,$ two important examples of posets of
dimension $n$ are:
\begin{equation}\begin{minipage}[c]{35pc}\label{d.2^n}
The $\!n\!$-cube $\2^n,$ i.e., the $\!n\!$-fold
direct product $\2\times\dots\times \2,$ where
$\2$ denotes the poset $\{0,1\}$ with the ordering $0 < 1,$
\end{minipage}\end{equation}
and
\begin{equation}\begin{minipage}[c]{35pc}\label{d.S_n}
For $n\geq 2,$ the ``standard example'' $S_n,$ whose underlying set
consists of $2n$ elements $\{a_0,\dots,a_{n-1}, b_0,\dots,b_{n-1}\},$
with the ordering that makes each $a_i$ less than every $b_j$
{\em other} than $b_i,$ and with no other order-relations between
distinct elements.
\end{minipage}\end{equation}

For $n\geq 3,$ $S_n$ can be identified with a subposet of $\2^n,$ by
identifying each $a_i$ with the $\!n\!$-tuple that has
value $1$ in the $\!i\!$-th coordinate and $0$
elsewhere, and each $b_i$ with the $\!n\!$-tuple that has
value $0$ in the $\!i\!$-th coordinate and $1$ elsewhere.
$(S_n$ is sometimes called the ``crown'' of dimension $n.$
This nicely fits the appearance of the diagram when
$n=3,$ but not so nicely for higher $n,$ or $n=2.$
In~\cite{F+T+W}, a different $\!2n\!$-element subposet of
$\2^n$ is called a ``crown''
$C_n$~\cite[Fig.~6, last two diagrams]{F+T+W},
which agrees with $S_n$ only for $n=3.)$

To see that~\eqref{d.2^n} and~\eqref{d.S_n} both have dimension $n,$
note first, for $n\geq 3,$ the inequalities
\begin{equation}\begin{minipage}[c]{35pc}\label{d.Sn2nn}
$\dim(S_n)\ \leq\ \dim(2^n)\ \leq\ n,$
\end{minipage}\end{equation}
which are clear in view of the representation of $S_n$ as a subposet of
$\2^n,$ and of $\2^n$ as a direct product of $n$ chains.
So to get the desired equalities
when $n\geq 3,$ it will suffice to show that the ordering on $S_n$ is
not an intersection of fewer than $n$ total orderings.

To see this, consider any family of total orderings of $|S_n|$
whose intersection is the ordering of $S_n.$
Note that for each $i\in n,$ since $a_i\nleq b_i$ in $S_n,$
that family must have at least one member
$\leq_i$ such that $a_i>_i b_i.$
I claim these orderings $\leq_i$ must be distinct.
Indeed, if for some $i\neq j,$
$\leq_i$ is the same as $\leq_j,$ let us
call their common value $\leq_{i,j}\!.$
Interchanging the roles of $i$ and $j$ if necessary, we may
assume $a_j>_{i,j} a_i.$
Then $a_j >_{i,j} a_i >_{i,j} b_i,$ hence under the intersection
of our family of orderings we have
$a_j\nleq b_i,$ contradicting the definition of $S_n.$
So the orderings $\leq_i$ $(i\in n)$ in our family are indeed distinct,
so $\dim(S_n)\geq n,$ so by~\eqref{d.Sn2nn}, both~\eqref{d.2^n}
and~\eqref{d.S_n} are $\!n\!$-dimensional.

As for the cases where $n<3:$ if $n=2,$ since neither
poset is a chain, both $\dim(S_2)$ and $\dim(\2^2)$ are $>1,$
and the latter is $\leq 2$ by definition.
With the help of the embedding of $S_2$ in $\mathbf{3}^2$ as
$\{(0,1), (1,0), (2,0), (0,2)\},$ we see that its
dimension is also $\leq 2;$ so both are indeed equal to~$2.$
For $n=1,$ $S_1$ is undefined, and $\dim(\2^1)=1$ is clear.

For another class of examples, recall
that a poset $P$ in which every pair of distinct
elements is incomparable is called an {\em antichain}.
We observe that
\begin{equation}\begin{minipage}[c]{35pc}\label{d.dim_antichain}
Every antichain $P$ with more than one element has dimension~$2.$
\end{minipage}\end{equation}
Indeed, if one chooses any total ordering of $|P|,$ then the
intersection of that ordering and the opposite ordering is
the antichain ordering, so $\dim(P)\leq 2;$
and since $|P|$ has more than one
element, no single linear ordering makes it an antichain.

We remark that under a bijective isotone map of posets which is
{\em not} an isomorphism,
the dimension may increase, decrease, or remain unchanged.
For instance, starting with an antichain of $8$ points,
a bijection onto the poset $\2^3$ is an isotone map
that increases the dimension
from $2$ to $3,$ while a linearization decreases the latter dimension
from $3$ to $1.$
On the other hand, a bijection from a $\!4\!$-element antichain to
the poset $\2^2$ is isotone and leaves the dimension, $2,$ unchanged.

Let us note one more elementary poset construction
(a generalization of the one we applied to chains
in proving the equivalence of the two conditions
in Definition~\ref{D.dim}).
Given a poset $P,$ and for each $p\in P$ a poset $Q_p,$
we may define the set
\begin{equation}\begin{minipage}[c]{35pc}\label{d.sum_set}
$|\sum_P Q_p| \ = \ \{(p,q)~|~p\in P,~q\in Q_p\},$
\end{minipage}\end{equation}
and partially order this set lexicographically; that is, by defining
\begin{equation}\begin{minipage}[c]{35pc}\label{d.sum_ord}
$(p,q)~\leq~(p',q')$ if and only if either $p<_P p',$
or $p=p'$ and $q\leq_{Q_p} q'.$
\end{minipage}\end{equation}
It is not hard to verify that
\begin{equation}\begin{minipage}[c]{35pc}\label{d.dim_sum}
$\dim(\sum_P Q_p)$ is the supremum of
$\{\dim(P)\} \cup \{\dim(Q_p)\,|\,p\in P\}.$
\end{minipage}\end{equation}

In the verification, one takes a representation
of $\leq_P$ as in Definition~\ref{D.dim}(i), and
representations of the posets $Q_p$ as in Definition~\ref{D.dim}(ii)
(which may, but need not, have the more restricted
form of condition~(i)).
By repeating mappings if necessary, one can assume that the number of
chains used in each of these representations is the supremum
indicated above.
One can then easily combine these (as in the paragraph
following Definition~\ref{D.dim}) to get an embedding
of $\sum_P Q_p$ in a product of that number of chains.
This gives ``$\leq$'' in~\eqref{d.dim_sum};
``$\geq$'' is straightforward.

\section{Our main result, and some consequences}\label{S.main}

To lead up to our main result, let us first note
that for any two posets $P$ and $Q$ (which, we recall,
are required to be nonempty), we clearly have
\begin{equation}\begin{minipage}[c]{35pc}\label{d.sup_<dim_<sum}
$\r{max}(\dim(P),\,\dim(Q))\ \leq\ \dim(P\times Q)\ %
\leq\ \dim(P)+\dim(Q).$
\end{minipage}\end{equation}

Though the properties of vector-space dimension
suggest that the second inequality
should be equality, the two concepts of dimension
are not alike in that respect.
To see an interesting
way that equality can fail, consider a poset $P$ of
dimension $d,$ represented as a subposet of a
product of $d$ chains, $P\subseteq \prod_{i\in d} T_i,$
and consider a nontrivial chain $C.$
Suppose that for each $i$ we let $T^*_i = T_i\ltimes C,$
the product-set ordered
lexicographically -- intuitively, the chain gotten from $T_i$ by
replacing each element $t$ by a miniature copy of $C.$
It is easy to see that the map $f:P\times C\to\prod_{i\in d} T^*_i$
taking $((p_0,\dots,p_{d-1}),c)$ to $((p_0,c),\dots,(p_{d-1},c))$
is one-to-one and isotone.
Can it fail to be an embedding?

In general, yes.
If $p< p'$ in $P$ and $c> c'$ in $C,$ then in the
product poset $P\times C,$
the elements $(p,c)$ and $(p',c')$ are by definition incomparable;
but if it happens that for all $i,$ we have the strict
inequality $p_i < p'_i,$
then under the ordering described above, $f(p,c)< f(p',c');$
so their images are not incomparable.

However, {\em if} the subposet $P\subseteq \prod_{i\in d} T_i$
has the property that whenever two elements $p,\,p'$ satisfy
$p<p',$ then as members of $\prod_{i\in d} T_i$ they
{\em agree} in at least one coordinate,
then looking at that coordinate, we see that we do get
incomparability between the indicated elements of $f(P\times C),$
and it is easy to check that $f$ is an embedding; so
we indeed have $\dim(P\times C)=d = \dim(P).$

The condition that every pair of comparable elements
of $P\subseteq \prod_{i\in d} T_i$ agree
in at least one coordinate may seem unnatural, but
for $d\geq 3$ it is easy to see that it holds in the poset $S_d,$
regarded as a subposet of $\2^d$ (sentence
following~\eqref{d.S_n}).

In fact, more is true in that example: Every pair of comparable
elements $a_i<b_j$ agrees in one coordinate where their common
value is $0$ (the $\!j\!$-th)
and one where their common value is $1$ (the $\!i\!$-th);
and the above construction can be adapted to show, as a
consequence, that taking the product of $S_d$ with {\em two}
chains does not increase its dimension.

The next result gives this argument in detail for a
more general situation, that applies to products
of a poset $P$ with possibly more than two chains.
Note, however, that in that result, $d$ is not assumed to
be the dimension of $P$ (as was the case in the
above example), but simply to be an integer such that
$P$ is a subposet of a product of $d$ chains with certain properties.
(Some tools similar to this result appear in \mbox{\cite[\S3]{F+T+W}.)}

\begin{theorem}\label{T.main}
Let $d$ and $n$ be nonnegative integers, $(T_i)_{i\in d}$ a
$\!d\!$-tuple of chains, which for notational convenience we will
assume pairwise disjoint and $P$ a subposet of $\prod_{i\in d} T_i.$
Suppose there exists an $\!n\!$-tuple $(M_j)_{j\in n}$ of
pairwise disjoint subsets of $\bigcup_{i\in d} T_i$ such that
\begin{equation}\begin{minipage}[c]{35pc}\label{d.eq_coord}
For every comparable pair of elements $p\leq p'$
in $P,$ and every $j\in n,$ there exists $i\in d$ such
that the $\!i\!$-th coordinates of $p$ and $p'$ are equal,
and their common value is a member of $M_j.$
\end{minipage}\end{equation}

Then for any $\!n\!$-tuple of chains $(C_j)_{j\in n}$ we have
\begin{equation}\begin{minipage}[c]{35pc}\label{d.main}
$\dim(P\times\prod_{j\in n} C_j)\ \leq\ d.$
\end{minipage}\end{equation}
\end{theorem}

\begin{proof}
For notational simplicity, we may assume all the $C_j$ are equal
to a common chain $C,$ and that
$\bigcup_{j\in n} M_j = \bigcup_{i\in d} T_i.$
Indeed, given structures as in the statement of the theorem,
we may embed all the $C_j$ in a common chain $C,$ and if we
prove~\eqref{d.main} with the $C_j$ all replaced by $C,$
this will imply the same inequality for the original
choices of $C_j.$
Likewise, if we enlarge some of the $M_j$
(still keeping them disjoint) so that their
union becomes the whole set $\bigcup_{i\in d} T_i,$
then the hypothesis about pairs of elements $p<p'$ assumed
for the original choices of $M_j$ remains true.
So let us make those assumptions.

To obtain~\eqref{d.main}, we need an embedding of the poset
on the left-hand side, which we can now write
$P\times C^n,$ in a product of $d$ chains $T'_i$ $(i\in d).$
For this purpose we define
\begin{equation}\begin{minipage}[c]{35pc}\label{d.T'}
$T'_i\ =\ T_i\ltimes C,$ the set-theoretic product
of $T_i$ and $C,$ ordered lexicographically, i.e., so that
\end{minipage}\end{equation}
\begin{equation}\begin{minipage}[c]{35pc}\label{d.lex}
$(t,c)\ \leq\ (t',c')$ if and only if
$t<t',$ or $t=t'$ and $c\leq c'.$
\end{minipage}\end{equation}
Also, since the $M_j$ $(j\in n)$ partition
$\bigcup_{i\in d} T_i,$ we can set the notation
\begin{equation}\begin{minipage}[c]{35pc}\label{d.mu}
For each $t\in\bigcup_{i\in d} T_i,$ $m(t)\in n$ will denote the
unique value such that $t\in M_{m(t)}.$
\end{minipage}\end{equation}

We now define a map $f: P\times C^n \to \prod_{i\in d} T'_i$ as follows.
\begin{equation}\begin{minipage}[c]{35pc}\label{d.f}
For $p=(p_i)_{i\in d}\in P \subseteq\prod_{i\in d} T_i,$
and $c=(c_j)_{j\in n}\in C^n,$ let
$f(p,c)$ be the element of $\prod_{i\in d} T'_i$ whose
$\!i\!$-th coordinate is $(p_i,c_{m(p_i)})$ for each $i\in d.$
\end{minipage}\end{equation}
We wish to show that $f$ is an embedding of posets.

First, $f$ is isotone, i.e.,
\begin{equation}\begin{minipage}[c]{35pc}\label{d.f_isotone}
If $(p,c) \leq (p',c')$ in $P\times C^n,$ then
$f(p,c) \leq f(p',c')$ in $\prod_{i\in d} T'_i.$
\end{minipage}\end{equation}
This follows easily from~\eqref{d.lex} and~\eqref{d.f}.

To complete the proof that $f$ is an embedding
(and in particular, is one-to-one), we need to show that
\begin{equation}\begin{minipage}[c]{35pc}\label{d.notleq}
If $(p,c)\,\nleq\,(p',c'),$ \,then\, $f(p,c)\,\nleq\,f(p',c').$
\end{minipage}\end{equation}

This breaks down into two cases.
Suppose first that
\begin{equation}\begin{minipage}[c]{35pc}\label{d.p_notleq}
$p\ \nleq\ p'.$
\end{minipage}\end{equation}

In that case, for some $i$ we have $p_i\nleq p'_i,$
i.e., $p_i> p'_i,$ and
looking at the $\!i\!$-th coordinates of $f(p,c)$ and $f(p',c'),$
namely $(p_i,c_{m(p_i)})$ and $(p'_i,c'_{m(p'_i)}),$
we see from the lexicographic ordering of $T'_i$
that the former is $>$ the latter,
giving the conclusion of~\eqref{d.notleq}.

If, on the other hand,
\begin{equation}\begin{minipage}[c]{35pc}\label{d.p_leq}
$p\,\leq\,p',$ \,but\, $c\,\nleq\,c',$
\end{minipage}\end{equation}
then let us choose a $j\in n$ such that
\begin{equation}\begin{minipage}[c]{35pc}\label{d.c_j_notleq}
$c_j\,>\,c'_j.$
\end{minipage}\end{equation}
Now by the {\em first} condition of~\eqref{d.p_leq} and
the hypothesis~\eqref{d.eq_coord}, there is some $i\in d$ such that
\begin{equation}\begin{minipage}[c]{35pc}\label{d.use_eq_coord}
$p_i\,=\,p'_i\in M_j.$
\end{minipage}\end{equation}
Note that by~\eqref{d.f},
\begin{equation}\begin{minipage}[c]{35pc}\label{d.ith_coords}
$f(p,c)$ and $f(p',c')$ have $\!i\!$-th terms
$(p_i,c_j)$ and $(p'_i,c'_j)$ respectively.
\end{minipage}\end{equation}

Since the $\!T_i\!$-coordinates of the above two $\!i\!$-th terms
are the same by~\eqref{d.use_eq_coord}, the order-relation
between them is
that of the $\!C\!$-coordinates, which satisfy~\eqref{d.c_j_notleq}.
This completes the proof of~\eqref{d.notleq}, and of the Theorem.
\end{proof}

I came up with the above result after pondering~\cite{S3x2x2},
which showed by an explicit construction that
$\dim(S_3\times \2\times \2) = 3,$ i.e., is the same as $\dim(S_3).$
The next corollary includes that case.

\begin{corollary}\label{C.0&1}
Suppose $d\geq 2$ and $(T_i)_{i\in d}$ is a family of non-singleton
chains, each having a least element $0_i$ and a greatest element $1_i,$
and $P$ is a subposet of $\prod_{i\in d} T_i$ consisting of
elements each of which has at least one coordinate of
the form $0_i$ and at least one of the form $1_{i'}.$
Then for any two chains $C_0$ and $C_1,$ we have
\begin{equation}\begin{minipage}[c]{35pc}\label{d.dimPxC0xC1}
$\dim(P\times C_0\times C_1)\ \leq\ d.$
\end{minipage}\end{equation}

In particular, if $d\geq 3$
and $P$ is any subposet of $\2^d\setminus \{0,1\}$
containing the standard poset $S_d$ \textup{(}e.g., if $P=S_d),$
then for any two chains $C_0$ and $C_1,$
\begin{equation}\begin{minipage}[c]{35pc}\label{d.dim=d=dim}
$\dim(P\times C_0\times C_1)\ =\ d\ =\ \dim(P).$
\end{minipage}\end{equation}
\end{corollary}

\begin{proof}
In the context of the first assertion,
given $p\leq p'\in P,$ since $p$ has
at least one coordinate of the form $1_i,$
$p'\geq p$ must also
have $1_i$ as its $\!i\!$-th coordinate; and similarly,
since $p'$ has a coordinate $0_{i'},$ $p$ must agree with
$p'$ in that coordinate.
Applying Theorem~\ref{T.main} with $n=2,$ $M_0=\{0_i\,|\,i\in d\}$
and $M_1=\{1_i\,|\,i\in d\},$ we get~\eqref{d.dimPxC0xC1}.

The second statement follows because when each $T_i$ is
the $\!2\!$-element set $\{0_i,\,1_i\},$ the exclusion
of $0$ and $1$ from $P$ forces every element
to have at least one coordinate of
the form $1_i$ and one of the form $0_{i'};$
and the assumption that $P$ contains $S_d,$ which
we saw following~\eqref{d.S_n} has dimension $d,$ turns the
inequality~\eqref{d.dimPxC0xC1} into equality.
\end{proof}

If, instead of assuming that every element
has at least one  $\!0\!$-coordinate and
at least one  $\!1\!$-coordinate, we only assume one of these
conditions, it is easy to check that the analogous reasoning
gives a conclusion half as strong:

\begin{corollary}\label{C.0}
Suppose $d\geq 1,$ $(T_i)_{i\in d}$ is a family of chains each
having a least element, $0_i,$
and $P$ is a subposet of $\prod_{i\in d} T_i$ consisting of
elements each of which has at least one coordinate of the form $0_i.$
Then for any chain $C$ we have
\begin{equation}\begin{minipage}[c]{35pc}\label{d.dimPxC}
$\dim(P\times C)\ \leq\ d.$
\end{minipage}\end{equation}

In particular, if $P$ is any subposet of $\2^d\setminus \{1\}$
containing $S_d$ \textup{(}e.g., if $P=S_d\cup\{0\})$
then for any chain $C$ we have
\begin{equation}\begin{minipage}[c]{35pc}\label{d.dimPxC=}
$\dim(P\times C)\ =\ d\ =\ \dim(P).$
\end{minipage}\end{equation}

The analogous statements hold with least elements $0_i$ and $0$
everywhere replaced by greatest elements $1_i$ and $1.$\qed
\end{corollary}

To get examples of Theorem~\ref{T.main}
with $n>2,$ we shall use a different way of choosing $n$ subsets
$M_j$ of $\bigcup_{i\in d} T_i,$ based on the subscript $i$ rather
than the distinction between greatest and least elements of~$T_i.$
We will need a bit of notation.
\begin{equation}\begin{minipage}[c]{35pc}\label{d.PdA}
For every positive integer $d$ and every subset
$A\subseteq d\!+\!1=\{0,\dots,d\},$
$B_{d,A}$ will denote the subset of $\2^d$ consisting
of those elements in which the number of coordinates
of the form $1_i$ is a member of~$A.$
\end{minipage}\end{equation}
For instance, when $d\geq 3,$ $S_d\cong B_{d,\,\{1,\,d-1\}}.$

Then we have

\begin{corollary}\label{C.Pa,a+1}
For every positive integer $d,$ every integer $a$ with
$0\leq a\leq d-1,$ and every family of chains $(C_j)_{j\in n}$ with
\begin{equation}\begin{minipage}[c]{35pc}\label{d.2n<d}
$n\ \leq\ d/2,$
\end{minipage}\end{equation}
we have, in the notation of~\eqref{d.PdA},
\begin{equation}\begin{minipage}[c]{35pc}\label{d.dimPdaa+1}
$\dim(B_{d,\,\{a,\,a+1\}}\times\prod_{j\in n} C_j)\ \leq\ d.$
\end{minipage}\end{equation}
\end{corollary}

\begin{proof}
Again writing $\2^d$ as $\prod_{i\in d}\{0_i,1_i\},$ let us define
$n$ $\!4\!$-element subsets of $\bigcup_{i\in d}\,\{0_i,1_i\}$ by
\begin{equation}\begin{minipage}[c]{35pc}\label{d.M_j}
$M_j\ =\ \{0_{2j},\,1_{2j},\,0_{2j+1},\,1_{2j+1}\},$ $(j=0,\dots,n-1).$
\end{minipage}\end{equation}
Condition~\eqref{d.2n<d} guarantees that this formula indeed
describes subsets of $\bigcup_{i\in d}\{0_i,1_i\}.$

Now if $p\leq p'$ are elements of $B_{d,\,\{a,\,a+1\}},$ then the
cardinalities of the subsets of $d$ on which $p$ and $p'$ assume
values of the form $1_i$ (namely, $a$ or $a+1),$
differ by at most $1,$ hence the order relation between
them necessitates that they disagree at most one coordinate.
Since each $M_j$ contains both $0_i$ and $1_i$ for
{\em two} values of $i,$ one of those values of $i$ must have
the property that $p$ and $p'$ agree on the $\!i\!$-th coordinate.
Hence these subsets $M_j$ satisfy the hypotheses
of Theorem~\ref{T.main}, completing the proof.
\end{proof}

If the posets $B_{d,\,\{a,\,a+1\}}\subseteq \2^d$ had, like
$S_d,$ dimension $d,$ then the above result would give us,
for every $n,$ finite posets whose dimensions were not changed on
taking a direct product with $n$ chains.
However, such subposets of $\2^d$ generally have dimension less
than $d.$
There are many results in the literature obtaining
bounds on the dimensions of subposets of $\2^d;$
cf.~\cite{BD}, \cite{F+T+W}, \cite{ZF}, \cite{H+K+T},
\cite[Theorem~7.1]{K+T}, \cite{Ko+Ta},
\cite[Theorems~7, 10, 12 and 13]{lin}, and \cite[Theorem~2]{Trotter89}.
Let us obtain one such result here, as another consequence
of~Theorem~\ref{T.main}.
(In the statement and proof, we shall not need to look
at the disjoint union of the factors of
$\2^d,$ hence we shall not, as in the preceding
corollaries, treat these as disjoint sets $\{0_i,1_i\},$
but as the same set $\2=\{0,1\}.)$

\begin{corollary}\label{C.P<>}
Let $d$ be a positive integer, and let
\begin{equation}\begin{minipage}[c]{35pc}\label{d.P=}
$P \ = \ B_{d,\,[2,\,d-2]}
\setminus\{(0,\dots,0,1,1),\ (1,\dots,1,0,0)\}.$
\end{minipage}\end{equation}
\textup{(}Here $B_{d,\,[2,\,d-2]}$ is defined
as in~\eqref{d.PdA}, with $A=[2,\,d\!-\!2]=\{2,3,\dots,d\!-\!2\};$
$(0,\dots,0,1,1)$ denotes the $\!d\!$-tuple whose
first $d\!-\!2$ coordinates are $0$ and whose last two
are $1,$ and $(1,\dots,1,0,0)$ the corresponding
$\!d\!$-tuple with the roles of $0$ and $1$ reversed.\textup{)}
Then
\begin{equation}\begin{minipage}[c]{35pc}\label{d.dimP_<}
$\dim(P)\ \leq\ d-2.$
\end{minipage}\end{equation}

Hence the same upper bound holds for the dimension of
any subposet of $P;$
in particular, for any poset of the form $B_{d,\,\{a,\,b\}}$
with $3\leq a < b \leq d-3.$
\end{corollary}

\begin{proof}
I claim that as a subset of $\2^d,$
\begin{equation}\begin{minipage}[c]{35pc}\label{d.Psubset}
$P \ \subseteq \ B_{d-2,\,[1,\,d-3]}\times \2\times \2;$
\end{minipage}\end{equation}
in other words, that the first $d-2$ coordinates
of every $p\in P$ contain at least one $1$ and at least one $0.$
To see the former condition, note that being contained in
$B_{d,\,[2,\,d-2]},$
$p$ itself must have at least two coordinates $1.$
If none of these were among the first $d-2$ coordinates,
this would force $p=(0,\dots,0,1,1),$ but that element is
excluded in~\eqref{d.P=}.
The second condition follows in the same way.

Now by Corollary~\ref{C.0&1}, with $d-2$ in place of $d,$
and each $T_i$ taken to be $\2,$ we see that the
right-hand side of~\eqref{d.Psubset} has
dimension $\leq d-2,$ hence the same is true of its
subposet $P,$ proving~\eqref{d.dimP_<}

The final sentence of the corollary follows immediately.
\end{proof}

\section{An old result, and questions old and new}\label{S.questions}

Having seen that dimension of posets is {\em not} in
general additive on direct products, it is striking that in an important
class of cases, it is.
The theorem below was proved, for products over index sets of arbitrary
cardinality, in K.~Baker's unpublished undergraduate
thesis~\cite{KB}.
That proof is summarized in~\cite{K+T}
for pairwise products (from which the case of arbitrary
finite products follows), assuming the dimensions finite.
I give below a version of the proof (also for pairwise products of
finite-dimensional posets) which I find easier to follow.

\begin{theorem}\label{T.prod_bdd}
\textup{(\cite[p.9, Property~3]{KB}, \cite[p.179, last 11 lines]{K+T})}
Let $P$ and $Q$ be finite-dimensional posets, each
having a least element $0$ and a greatest element $1.$
Then
\begin{equation}\begin{minipage}[c]{35pc}\label{d.dim=sum}
$\dim(P\,\times\,Q)\ =\ \dim(P)\,+\,\dim(Q).$
\end{minipage}\end{equation}
\end{theorem}

\begin{proof}
In view of Definition~\ref{D.dim}(ii), $\leq$ clearly
holds in~\eqref{d.dim=sum}, so it suffices to show $\geq.$
For this we shall use Definition~\ref{D.dim}(i),
and show that if we have an expression for the partial order
of $P\times Q$ as an intersection of
\begin{equation}\begin{minipage}[c]{35pc}\label{d.leqi}
$n$ total orderings, \ $\leq_0,\ \dots,\ \leq_{n-1},$ \ %
on\ $|P\,\times\,Q|,$
\end{minipage}\end{equation}
then we can split that family of orderings into
two disjoint subsets, such that a certain map of $P$ into
the product of chains determined by one of those subsets is
an embedding, as is a map of
$Q$ into the product determined by the other.
Thus the former subset must consist of $\geq\dim(P)$
orderings and the latter of $\geq\dim(Q),$ giving
$n\geq\dim(P)+\dim(Q),$ as desired.

The trick to finding this partition is to look at the
relative order, under each of the orderings~\eqref{d.leqi},
of the elements $(0_P,1_Q)$ and $(1_P,0_Q)$ of $|P\times Q|.$
Reindexing those $n$ orderings if necessary, we may assume that
for some $m\leq n,$
\begin{equation}\begin{minipage}[c]{35pc}\label{d.(0,1)vs(1,0)}
$(0_P,1_Q) <_i (1_P,0_Q)$ \,for\, $0\leq i < m,$ \,while
$(1_P,0_Q) <_i (0_P,1_Q)$ \,for\, $m\leq i < n.$
\end{minipage}\end{equation}
Let us now show that the map
\begin{equation}\begin{minipage}[c]{35pc}\label{d.P->}
$P\to (|P\times Q|,\,\leq_0)\times\dots\times(|P\times Q|,\,\leq_{m-1})$
\ given by \ $p\mapsto ((p,0_Q),\dots,(p,0_Q)),$
\end{minipage}\end{equation}
which is clearly isotone, is an embedding;
i.e., that if $p,\,p'\in P$ satisfy
\begin{equation}\begin{minipage}[c]{35pc}\label{d.nleq}
$p\ \nleq\ p',$
\end{minipage}\end{equation}
then the corresponding condition holds on
the images of $p$ and $p'$ under~\eqref{d.P->}.

To show this, note that under the product ordering on $P\times Q,$
the relation~\eqref{d.nleq} implies that $(p,0_Q)\nleq (p',1_Q).$
Hence we can find some $i$ such that
\begin{equation}\begin{minipage}[c]{35pc}\label{d.nleq_i}
$(p,0_Q)\ \nleq_i\ (p',1_Q).$
\end{minipage}\end{equation}
Now if $i\geq m,$ we would have
$(p,0_Q)\leq_i (1_P,0_Q) <_i (0_P,1_Q) \leq_i (p',1_Q),$
contradicting ~\eqref{d.nleq_i}; so $i< m.$
Since~\eqref{d.nleq_i} implies $(p,0_Q)\nleq_i (p',0_Q),$
this completes the proof that~\eqref{d.P->} respects $\nleq,$
hence is an embedding.
The analogous argument shows that the map
$Q\to {(|P\times Q|,\leq_m)\times\dots\times(|P\times Q|,\leq_{n-1})}$
given by $q\mapsto ((0_P,q),\dots,(0_P,q))$ likewise
gives an embedding of $Q.$
As noted in the first paragraph of this
proof, this establishes~\eqref{d.dim=sum}.
\end{proof}

Though we have seen that~\eqref{d.dim=sum}
does not hold without the assumptions
that $P$ and $Q$ have upper and lower bounds, the deviations
from equality in all known examples are $\leq 2,$
leading to the longstanding open question:

\begin{question}\label{Q.bound_2}
\textup{(i)}\ \ If $P$ and $Q$ are finite-dimensional posets,
must $\dim(P\times Q)\geq \dim(P)+\dim(Q)-2$?

In particular,\\
\textup{(ii)}\ \ If $P$ is a finite-dimensional poset, and $n$
a positive integer such that $\dim(P\times \2^n)\ =\ \dim(P),$
must $n\leq 2$?
\end{question}

In~\cite[last line of p.179 and top line of p.180]{K+T}
it was conjectured that $\dim(P)+\dim(Q)$ can exceed $\dim(P\times Q)$
by at most the number of members
of the set $\{P,\,Q\}$ that do {\em not} have both a
greatest and a least element.
Thus, the case where that number is $0$ is
Theorem~\ref{T.prod_bdd}, while the case where there is no restriction
on $P$ or $Q$ corresponds to Question~\ref{Q.bound_2}(i) above.
The case of that conjecture
where that number is~$1,$ however, turned out to
be false: from Corollary~\ref{C.0&1} we see that
for any $n\geq 3,$ $\dim(S_n\times \2^2)=\dim(S_n),$
so $\dim(S_n)+\dim(\2^2)$ exceeds
$\dim(S_n\times \2^2)=\dim(S_n)$ by $2,$
though $\2^2$ has both greatest and least elements.

However, one might ask about different intermediate cases
between those of Theorem~\ref{T.prod_bdd} and
Question~\ref{Q.bound_2}(i),
suggested by Corollary~\ref{C.0}, which we note as
Question~\ref{Q.bound_1} below.
C.\,Lin poses part~(iii) of that question
in~\cite{lin}, though Theorem~10 and Lemma~11 of that paper
suggest the stronger implications of~(i) and~(ii).

\begin{question}\label{Q.bound_1}
\textup{(i)}~~If $P$ and $Q$ are finite-dimensional posets
such that $P$ has a least element and
$Q$ has a greatest element,
must $\dim(P\times Q)\geq \dim(P)+\dim(Q)-1$?

\textup{(ii)}~~If $P$ and $Q$ are finite-dimensional posets
each of which has a least element,
or each of which has a greatest element,
must $\dim(P\times Q)\geq \dim(P)+\dim(Q)-1$?

And finally, a possible implication weaker than either of the above two:

\textup{(iii)}~~\cite[p.80]{lin}~~If $P$ and $Q$ are
finite-dimensional posets such that $P$ has a least or a
greatest element, and $Q$ has both,
must $\dim(P\times Q)\geq \dim(P)+\dim(Q)-1$?
\end{question}

\cite[Theorem~$3'$]{KR90}$=$\cite[Theorem~10]{F+T+W} shows
that~(ii) has a positive answer (in fact, that equality holds)
in the case where $P$ is an arbitrary finite poset with a least
element, while $Q$ is the $\!3\!$-element poset (in those
papers called $\mathbf{V}$ or $V)$ consisting of a least element and
two mutually incomparable elements above it.
(And, of course, this implies the corresponding
result with all posets turned upside down.)

Returning to Question~\ref{Q.bound_2}, I have wondered whether
in Theorem~\ref{T.prod_bdd}, where no dimensionality
was lost in forming a product of posets, this could be related to the
hypothesis that the sets of minimal and of maximal elements of each
factor were singletons, i.e., had dimension~$0.$
If so, could one use the fact that for arbitrary finite~$P$
and~$Q$ the sets of minimal and maximal
elements are antichains, hence have dimension~$\leq 2,$
to get a positive answer to Question~\ref{Q.bound_2} for
finite posets?
But I see no way to adapt the proof of Theorem~\ref{T.prod_bdd}.

Moving on to other questions, recall that
in Theorem~\ref{T.main}, the result did not depend on
the lengths of the chains $C_j.$
This suggests

\begin{question}\label{Q.2_vs_C}
\textup{(}Cf.~\cite{indep_of_n}, \cite[Conjecture 1]{KR}\textup{)}
If $P$ is a finite-dimensional poset, and $C$ any chain of
more than one element, must $\dim(P\times C) = \dim(P\times \2)$?
\end{question}

The next question at first seemed ``obviously'' to have
an affirmative answer; but I see no argument proving~it.

\begin{question}\label{Q.Px2x2}
If $P$ is a finite-dimensional poset,
and $\dim(P\times \2) = \dim(P) + 1,$
must $\dim(P\times\nolinebreak \2\times\nolinebreak \2) = \dim(P) + 2$?
\end{question}

Much more generally (but much more vaguely), one may ask

\begin{question}\label{Q.P123}
Given finite-dimensional posets
$P_0,$ $P_1,$ $P_2,$ if we know their dimensions,
and those of $P_0\times P_1,$ $P_0\times P_2$ and $P_1\times P_2,$
what can we say about $\dim(P_0\times P_1\times P_2)$?
\textup{(}Anything more than that it is greater than
or equal to the dimensions of each of
the pairwise products, and less than or equal to the values
$\dim(P_i\times P_j)+\dim(P_k)$ for
all choices of $\{i,\,j,\,k\}=\{0,1,2\}$?\textup{)}
\end{question}

One example of ``misbehavior'' on this front:
Suppose $P$ is an antichain of more than one element, and
$P'=P\times \2.$
Now $P,$ $P^2,$ and $P^3$ are all
antichains, hence all have dimension~$2;$
but $P',$ $P'^2,$ and $P'^3$
have the forms $P\times\2,$ $P^2\times \2^2,$ and $P^3\times \2^3,$
and from~\eqref{d.dim_sum} we see that they
have dimensions $2,$ $2,$ and $3$ respectively.
So the dimensions of three
posets and of their pairwise direct products do not
determine the dimension of the product of all three
(the cases contrasted being where all three posets are $P,$
and where all three are $P').$

\section{Absorbency}\label{S.abs}

The following concept might be helpful
in studying questions of the sort we have been considering.

\begin{definition}\label{D.abs}
If $P$ is a finite-dimensional poset, let us define its
{\em absorbency} to be
\begin{equation}\begin{minipage}[c]{35pc}\label{d.abs}
$\abs(P)\ =$ the largest natural number $n$ such that
$\dim(P\times\prod_{i\in n} T_i) = \dim(P)$ for every
$\!n\!$-tuple of chains $(T_i)_{i\in n}.$
\end{minipage}\end{equation}
\end{definition}
As an example,
\begin{equation}\begin{minipage}[c]{35pc}\label{d.absS_n}
For all $d\geq 3,$ $\abs(S_d)\ =\ 2.$
\end{minipage}\end{equation}
Here $\geq$ follows from the $P=S_d$ case of~\eqref{d.dim=d=dim}.
To get the reverse inequality, we will use a result from the literature.
Note that if $\abs(S_d)$ were $\geq 3,$ then since $S_d\subseteq \2^d,$
the value of $\dim(S_d\times S_d)$ would be
$\leq \dim(S_d\times \2^d)\leq \dim(S_d\times \2^3)+\dim(\2^{d-3})
= d+(d-3)=2d-3.$
However, Trotter shows in~\cite[Theorem~2]{Trotter89}
that for all $d\geq 3,$ $\dim(S_d\times\nolinebreak S_d)=2d-2.$

An immediate property of absorbency is
\begin{equation}\begin{minipage}[c]{35pc}\label{d.absleqdim}
$\abs(P)\ \leq\ \dim(P),$
\end{minipage}\end{equation}
since for any $n>\dim(P),$ the product of $P$ with, say, $\2^n$ will
contain a copy of $\2^n,$
hence have dimension greater than $\dim(P).$

Here are some other easy properties.

\begin{lemma}\label{L.abs}
Let $P$ and $Q$ be finite-dimensional posets.
\begin{equation}\begin{minipage}[c]{36pc}\label{d.dimPgeqabsQ}
If $\abs(P) \geq \dim(Q),$ \ then \ $\dim(P\times Q) = \dim(P),$
and $\abs(P\times Q) \geq \abs(P)-\dim(Q)+\abs(Q).$
\end{minipage}\end{equation}

The situation where neither the hypothesis of~\eqref{d.dimPgeqabsQ},
nor the corresponding hypothesis with the roles of
$P$ and $Q$ reversed holds is covered by
\begin{equation}\begin{minipage}[c]{35pc}\label{d.dimgeqabs}
If $\dim(Q) \geq \abs(P)$ and $\dim(P) \geq \abs(Q),$ \ then\\[.2em]
$\max(\dim(P),\,\dim(Q))\ \leq\ \dim(P\times Q)\ \leq\ %
\dim(P)+\dim(Q)-\max(\abs(P),\,\abs(Q)).$
\end{minipage}\end{equation}
\end{lemma}

\begin{proof}
In~\eqref{d.dimPgeqabsQ}, the first conclusion holds
because $Q$ embeds in a product of $\dim(Q)\leq\abs(P)$ chains,
which $P$ can ``absorb'' without increasing its dimension.

To get the final inequality of~\eqref{d.dimPgeqabsQ},
we must show that the dimension
of $P\times Q$ is not increased on multiplying it by a
product of $\abs(P)-\dim(Q)+\abs(Q)$ chains.
Let us write such a product as $X\times Y,$ where
$X$ is a product of $\abs(P)-\dim(Q)$ chains,
and $Y$ a product of $\abs(Q)$ chains.
Then if we write $P\times Q\times (X\times Y)$ as
$P\times(X\ \times (Q\times Y)),$ we see that
$Q\times Y$ has dimension $\dim(Q)$ (by our choice of $Y),$ hence
$X\ \times (Q\times Y)$ has (by our choice of $X)$ dimension at most
$(\abs(P)-\dim(Q))+\dim(Q) = \abs(P),$ so its
product with $P$ has dimension $\dim(P),$ which
we have noted equals $\dim(P\times Q),$ so we have
indeed shown that the dimension
of $P\times Q$ has not been increased, as required.

In the conclusion of~\eqref{d.dimgeqabs}, the left-hand inequality
is immediate.
The right-hand inequality is equivalent to saying that $\dim(P\times Q)$
is bounded above by both $\dim(P)+\dim(Q)-\abs(P)$
and $\dim(P)+\dim(Q)-\abs(Q),$ so by symmetry it suffices to
prove the former bound.
If we embed $Q$ in a product of $\dim(Q)$ chains,
and break this into a product $X$ of $\abs(P)$ chains
and a product $Y$ of $\dim(Q)-\abs(P)$ chains, then
we have $\dim(P\times Q)\leq \dim(P\times X\times Y)\leq
\dim(P\times X) + \dim(Y) = \dim(P) + (\dim(Q)-\abs(P)),$
as desired.
\end{proof}

We may ask

\begin{question}\label{Q.abs}
Under the hypothesis of~\eqref{d.dimPgeqabsQ},
does equality always hold in the final inequality of that implication?
\end{question}

If we ask the same question about
the final inequality of~\eqref{d.dimgeqabs}, the peculiarities
noted in the final paragraph of \S\ref{S.questions} make trouble:
Using~\eqref{d.dim_sum}, one can show that the poset $P'$ defined
in that paragraph of \S\ref{S.questions} has absorbency $1,$
and deduce that if in~\eqref{d.dimgeqabs} we take both $P$ and $Q$
to be that poset $P',$ the final inequality of~\eqref{d.dimgeqabs}
becomes $2<3.$
But those peculiarities involved disconnected posets,
and I know of no cases not of that sort.
(A poset $P$ is called disconnected if it
can be written as a union of subposets
$P_0$ and $P_1$ such that all elements of
$P_0$ are incomparable with all elements of $P_1,$
and connected if it cannot be so written.)
We can thus ask

\begin{question}\label{Q.cnnctd}
If the posets $P$ and $Q$ in~\eqref{d.dimgeqabs}
are connected, must equality
hold in the final inequality of the conclusion thereof?
\end{question}

Since the case of~\eqref{d.dim_sum} where $P$ is an antichain
gives us a way of computing the
dimension of a disconnected poset from the dimensions of its
connected components, a positive answer to
Question~\ref{Q.cnnctd} (together with the
first assertion of~\eqref{d.dimPgeqabsQ}) would
allow us to compute exactly the dimension of a product of
two arbitrary finite-dimensional posets, given the dimensions and
absorbencies of their connected components.
It would, in particular, imply positive answers to
Questions~\ref{Q.2_vs_C} and~\ref{Q.Px2x2}.

But as long as we do
not know a positive answer to Question~\ref{Q.cnnctd},
here are some further ideas worth considering:

If Question~\ref{Q.2_vs_C} has a negative answer, one could
define variants of the absorbency concept, depending on
the lengths of the chains involved.

If Question~\ref{Q.2_vs_C} has a positive answer
but Question~\ref{Q.Px2x2} does not,
one might define the ``eventual absorbency''
of a finite-dimensional poset $P$ as the supremum, as
$n\to\infty,$ of $(\dim(P) + n) - \dim(P\times \2^n).$

Given a poset $P$ and a nonnegative integer
$d$ about which we know that $\dim(P)\leq d,$ one might define
the absorbency of $P$ {\em relative} to $d$ to
be the largest $n$ such that the product of $P$ with every
$\!n\!$-tuple of chains continues to have dimension $\leq d.$
So, for instance, given finite-dimensional posets $P$ and $Q,$
the absorbency of $P\times Q$ relative to $\dim(P)+\dim(Q)$
will be at least the sum of the absorbencies of $P$ and of $Q.$

\section{Bounded dimension}\label{S.bd-dim}

Let us call a poset {\em bounded} if it has a greatest
and a least element.

(I wish I knew a better term for
this condition, since it has very different properties from
the familiar sense of ``bounded''.
E.g., a subposet of a bounded poset is not, in general, bounded.
However, ``bounded'' is used in this way in places in the literature;
e.g., \cite[p.179, 12th line from bottom]{K+T}.)

This allows us to define
another function which is useful in studying dimensions:

\begin{definition}\label{D.bd-dim}
For $P$ a finite-dimensional poset, the
{\em bounded dimension} of $P,$
denoted $\r{bd\mhyphen dim}(P),$ will be defined to be

\textup{(i)}~~the greatest integer $n$ such that $P$ contains a
bounded subposet $P'$ satisfying $\dim(P')=n,$\\
equivalently,

\textup{(ii)}~~the maximum, over all pairs of elements $p\leq p'$
in $P,$ of $\dim(\{x~|~p\leq x\leq p'\}).$
\end{definition}

Since every bounded subset $P'$ of $P$ is contained in an
interval $\{x~|~p\leq x\leq p'\},$ we see that
the two versions of the above definition are indeed equivalent.
Because of Theorem~\ref{T.prod_bdd},
this function behaves very nicely under direct products:

\begin{lemma}\label{L.bd-dim}
For finite-dimensional posets $P$ and $Q,$
\begin{equation}\begin{minipage}[c]{35pc}\label{d.bd-dim}
$\r{bd\mhyphen dim}(P\times Q)\ =\ %
\r{bd\mhyphen dim}(P)\ +\ \r{bd\mhyphen dim}(Q).$
\end{minipage}\end{equation}
\end{lemma}

\begin{proof}
For elements $p\leq p'$ of any poset $P,$ let us write
\begin{equation}\begin{minipage}[c]{35pc}\label{d.[p,p']}
$[\,p,p']\ =\ \{\,p''\in P~|~p\,\leq\,p''\leq\,p'\},$ regarded as
a subposet of $P.$
\end{minipage}\end{equation}

Thus, by version~(ii) of the definition of bounded dimension,
$\r{bd\mhyphen dim}(P\times Q)$ is the greatest of the
values $\dim([(p,q),\,(p',q')])$ for $(p,q)\leq (p',q')$
in $P\times Q.$

But by the order-structure of a direct product,
$[(p,q),\,(p',q')]$ is isomorphic to $[\,p,\,p']\times [\,q,\,q'],$
and by Theorem~\ref{T.prod_bdd} the dimension of this
product is $\dim([\,p,p'])+\dim([\,q,q']).$
Taking the maximum over all pairs $p\leq p'$ and $q\leq q',$
we get $\r{bd\mhyphen dim}(P)+\r{bd\mhyphen dim}(Q).$
\end{proof}

Clearly, for any poset $P,$
\begin{equation}\begin{minipage}[c]{35pc}\label{d.bd,d}
$\r{bd\mhyphen dim}(P)\ \leq\ \dim(P).$
\end{minipage}\end{equation}
However, the two sides  of~\eqref{d.bd,d} can be far from equal, as
can be seen by taking $P=S_d$ for large $d.$
Then the left-hand side of~\eqref{d.bd,d}
is easily seen to be~$1,$ while the right-hand side is~$d.$

The concept of bounded
dimension allows us to get an upper bound on absorbency:
\begin{equation}\begin{minipage}[c]{35pc}\label{d.abs_leq}
$\r{abs}(P)\ \leq\ \dim(P)\,-\,\r{bd\mhyphen dim}(P).$
\end{minipage}\end{equation}

Indeed, if we take the product of $P$ with more
nontrivial chains than the number on the right,
then since $P$ contains a bounded set of dimension
$\r{bd\mhyphen dim}(P),$ and each of the chains we multiply by contains
a nontrivial bounded chain, the whole product will have dimension
$>\r{bd\mhyphen dim}(P) + (\dim(P)-\r{bd\mhyphen dim}(P))=\dim(P).$
(Incidentally, this argument is easily adapted to give the same upper
bound for the possibly larger ``eventual absorbency'' function
defined in the next-to-last paragraph of the preceding section.)

The two sides of~\eqref{d.abs_leq} can also be far from equal.
Indeed, again taking $P=S_d,$ where $d\geq 3,$
we saw in~\eqref{d.absS_n} that
the left-hand side of~\eqref{d.abs_leq} is~$2.$
On the other hand, the right-hand side is $d-1.$

From~\eqref{d.abs_leq} we get a strengthening
of~\eqref{d.absleqdim}; namely, we can add
to that inequality the observation,
\begin{equation}\begin{minipage}[c]{35pc}\label{d.when_abs=dim}
The only finite-dimensional posets $P$ for which
$\abs(P)=\dim(P)$ are the antichains,
\end{minipage}\end{equation}
since a poset that is not an antichain contains a copy of $\2,$
and so has bounded dimension at least~$1.$

Trotter~\cite[Conjecture~2]{Trotter89} suggested that
for all $n\geq 2$ there exist posets $P$ of dimension
$n$ such that $P\times P$ also has dimension~$n.$
Reuter \cite[Theorem~13]{KR} showed that no such $P$ exists for $n=3;$
but it is conceivable that there exist such $P$ for higher $n.$
Such an example would imply a negative answer to
Question~\ref{Q.cnnctd}, in view of~\eqref{d.when_abs=dim}.

It might also be of interest to look at the variants
of the bounded-dimension function in which {\em bounded} is replaced by
{\em bounded above} or {\em bounded below}.

\section{Boolean dimension}\label{S.Boolean}

Another variant of the dimension function, studied,
inter alia, in~\cite{comparing}, \cite{5/6}, \cite{G+N+T},
\cite{boolean}, \cite{T+W}, is based on the concept of a
{\em Boolean representation} of a poset~$P.$
Here one considers families of $d>0$ total orderings
$<_0,\dots,<_{d-1}$ of $|P|,$ which are {\em not} assumed
to be strengthenings of the partial
ordering $<_P$ of $P,$ but merely to have the
property that whether a pair of elements $x\neq y\in P$
satisfies $x<_P y$ is a function of the
set $\{i\in d~|~x<_i y\}.$
For example, the order relation of $S_n$ $(n\geq 3)$
can be so described in terms of the four total orderings of $|S_n|,$
\begin{equation}\begin{minipage}[c]{35pc}\label{d.Sn0}
$a_0 <_0 \dots <_0 a_i <_0 \dots <_0 a_{n-1} <_0
b_0 <_0 \dots <_0 b_i <_0 \dots <_0 b_{n-1},$
\end{minipage}\end{equation}
\begin{equation}\begin{minipage}[c]{35pc}\label{d.Sn1}
$b_0 <_1 \dots <_1 b_i <_1 \dots <_1 b_{n-1} <_1
a_0 <_1 \dots <_1 a_i <_1 \dots <_1 a_{n-1},$
\end{minipage}\end{equation}
\begin{equation}\begin{minipage}[c]{35pc}\label{d.Sn2}
$a_0 <_2 b_0 <_2 \dots <_2 a_i <_2 b_i <_2
\dots <_2 a_{n-1} <_2 b_{n-1},$
\end{minipage}\end{equation}
\begin{equation}\begin{minipage}[c]{35pc}\label{d.Sn3}
$b_0 <_3 a_0 <_3 \dots <_3 b_i <_3 a_i <_3
\dots <_3 b_{n-1} <_3 a_{n-1}.$
\end{minipage}\end{equation}
Namely, $x < y$ in $S_n$ if and only
if on the one hand, $x$ precedes $y$
in~\eqref{d.Sn0} but follows it in~\eqref{d.Sn1} (which together
say that $x$ has the form $a_i$ and $y$ the form $b_j),$
and, further, the relative order of $x$ and
$y$ is the same in~\eqref{d.Sn2} as in~\eqref{d.Sn3} (which says
that $(x,y)$ is {\em not} a pair of the form $(a_i,\,b_i)).$
This is called a {\em Boolean representation} of $S_n$ in
terms of the four total orderings~\eqref{d.Sn0}-\eqref{d.Sn3}.
The {\em Boolean dimension} of a poset $P,$ $\r{bdim}(P),$
is the least $d$ such
that $P$ has a Boolean representation in terms of $d$ total orderings.
This number is always $\leq\dim(P),$ but often strictly less;
e.g., the above example shows that for all $n,$ $\r{bdim}(S_n)\leq 4.$

In the formal definition of a Boolean representation
of $<_P$ in terms of total orderings $<_0,\dots,<_{d-1},$
one begins by mapping $\{(x,y)\in|P|\times |P|\ |\ x\neq y\}$
to $|\2^d|$ by sending $(x,y)$ to the $\!d\!$-tuple which
has $1$ in the $\!i\!$-th position if and only if $x<_i y.$
The ordering of $P$ is then
determined by a function $\tau: |\2^d|\to\{0,1\},$
such that $x<_P y$ if and only if~$\tau$
takes the $\!d\!$-tuple determined by $(x,y)$ to~$1.$
For instance, in the above example describing $S_n,$ $\tau$ is
the function on $|\2^4|$ taking the value~$1$ only
at the two $\!4\!$-tuples $(1,0,0,0)$ and $(1,0,1,1).$
Here the conditions that for a $\!4\!$-tuple to be taken to $1,$
its first entry must be $1$ and its second entry $0$
translate the conditions that
$x$ must precede $y$ in~\eqref{d.Sn0} but not in~\eqref{d.Sn1};
and the condition that the last two entries of the $\!4\!$-tuple
be equal translates the condition that $x$ and $y$ must occur
in the same order in~\eqref{d.Sn2} and in~\eqref{d.Sn3}.
The term ``Boolean dimension'' refers to the fact that any set-map
$\tau: |\2^d|\to 2$ can be expressed by a Boolean word
in $d$ variables.

In~\cite{T+W} it is shown that $\r{bdim}(P)$ agrees with $\dim(P)$
when the latter is~$3,$ and in \cite[Lemma~4]{5/6} the inequality
\begin{equation}\begin{minipage}[c]{35pc}\label{d.bdim_prod}
$\r{bdim}(P\times Q)\ \leq\ \r{bdim}(P)\,+\,\r{bdim}(Q)$
\end{minipage}\end{equation}
is proved.
On the other hand, \mbox{\cite[Lemma~6]{5/6}} shows that
$\r{bdim}(\2^6)=5,$ whence $\r{bdim}(\2^d)\leq\lceil 5d/6\rceil$ for
all $d$~\cite[Theorem~5]{5/6}.
However, \mbox{\cite[Proposition~9]{5/6}} shows that
for each $d,$ the $\!d\!$-th power of a large enough
finite chain has Boolean dimension~$d.$

A complication in this whole subject
is that there are two versions of the definition
of Boolean dimension in the literature, which {\em almost}
agree, but not quite.
Continuing to write $\r{bdim}(P)$ for the dimension function
defined as above, the other, which I
shall denote $\r{leq}\!$-$\!\r{bdim}(P),$
is defined as the least positive integer $d$ such that there exists
a $\!d\!$-tuple of linear orderings of $|P|$ with the property that
for {\em all} pairs $(x,y)$ of elements of $P$ (no longer required to
satisfy $x\neq y),$ the condition $x\leq_P y$ can be expressed as a
Boolean function of the $d$ conditions $x\leq_i y$ $(i\in d).$

It is easy to see that $\r{bdim}(P)\leq \r{leq}\!$-$\!\r{bdim}(P)\!:$
if the condition $\leq_P$ on arbitrary pairs can be expressed
as a Boolean function of the $d$ conditions $x\leq_i y,$
then restricting that Boolean function to pairs of distinct
elements $x\neq y,$ the conditions $x\leq_i y$ become $x<_i y,$ and
we get a Boolean realization of $<_P$ in terms of these relations.

Whether the reverse implication
holds is less obvious, since if we are given a
Boolean realization of $<_P$ in terms of $d$ linear
orderings $<_i,$ that Boolean function applied to the relations
$\leq_i$ might not give $1$ on pairs of the form $(x,x).$

In fact, that reverse implication does
hold, {\em except} when $P$ is a non-singleton antichain.
I do not know of any statement of this result in the literature.
(I am told that workers in the field are aware
of there being two slightly different versions of Boolean dimension,
but this is rarely mentioned in print.
The one point I know of where it is mentioned
is~\cite{comparing}, footnotes to pp.245 and~248.)
So let us prove this almost-equivalence result here.

The proof will use the observation that given any $\!d\!$-tuple
of total orderings $<_i$ $(i\in d)$ of a set $|P|,$ in terms of
which a partial ordering $<_P$ can be described on pairs
of distinct elements by a Boolean
expression, one can replace any subset of the orderings $<_i$ by
the opposite orderings, and still express $<_P$ in terms of the
resulting linear orders; this is done simply by inserting
appropriate ``not'' operators in the Boolean expression for $<_P.$
(The corresponding statement is not true for
expressions of $\leq_P$ in terms of the $\leq_i,$ since
on pairs $(x,y)$ not required to satisfy
$x\neq y,$ the relation $x \geq_i y$ is not
the negation of $x \leq_i y.)$

\begin{lemma}\label{L.leq-Bool}
For all finite posets $P$ other than non-singleton antichains,
$\r{bdim}(P)=\r{leq}\!$-$\!\r{bdim}(P).$

On the other hand, if $P$ is a non-singleton antichain,
then $\r{bdim}(P)=1,$ while $\r{leq}\!$-$\!\r{bdim}(P)=2.$
\end{lemma}

\begin{proof}
For the first assertion, we have noted that
$\r{bdim}(P)\leq\r{leq}\!$-$\!\r{bdim}(P),$ so what
we must prove is the reverse inequality.
A singleton poset can easily be seen to have both
dimensions~$1,$ so assume $P$ is a poset that is not an antichain,
and that $<_P$ can be written as a Boolean expression in
linear orderings $<_0,\dots,<_{d-1}$ on $|P|.$

Since $P$ is {\em not} an antichain, the Boolean expression in question
must assume the value $1$ on at least one string of $0$'s and $1$'s.
If we choose such a string, and
replace the orderings corresponding to the $0$'s (if any)
in that string with the opposite orderings, we get, as discussed in the
paragraph before the statement of this lemma, a realization of $<_P$
using a Boolean expression which now
takes the value $1$ on $(1,\dots,1).$

If we now apply this Boolean expression to the relations
$\leq_0,\dots,\leq_{d-1},$ the result will behave
on pairs $(x,y)$ with $x\neq y$ like our realization of $<_P,$
hence will agree there with $\leq_P,$ while on
pairs $(x,x),$ since these satisfy $x\leq_i x$ for all $i$
and our Boolean expression maps $(1,\dots,1)$ to $1,$
it will also agree with $\leq_P.$
This proves the first assertion of the lemma.

To get the second assertion, let $P$ be a non-singleton antichain.
We can establish $\r{bdim}(P)=1$ by choosing any total ordering $<_0$
of $|P|,$ and using as realizer the Boolean operation $0$
(making elements incomparable regardless of which is larger
under $<_0).$
This does not work for $\r{leq}\!$-$\!\r{bdim}(P)$ because it
would also give $x\nleq_P x$ for $x\in|P|;$ and it is equally
easy to see that none of the other three Boolean operations
in one variable work.
However, if we take any linear ordering $\leq_0$ on $|P|,$
and let $\leq_1$ be the opposite ordering,
we see that the relation $(x\leq_0 y)\wedge(x\leq_1 y)$
is equivalent to $x=y,$ and hence realizes the
antichain ordering $\leq_P;$ so $\r{leq}\!$-$\!\r{bdim}(P)=2.$
\end{proof}

(In the definition of $\r{bdim}(P),$ it is tempting to allow $d=0,$
which would come into play only when $P$ was an antichain,
in which case the element $0$ of the free Boolean ring
on zero generators, which is $\{0,1\},$ would indeed determine
the relation $x<_P y$ on pairs $x\neq y$ as never holding.
However, under this definition, the result~\eqref{d.bdim_prod} would
fail when one of $P,$ $Q$ was a nontrivial antichain and the other
was not an antichain -- the proof of~\eqref{d.bdim_prod} requires
$\r{bdim}(P)$ and $\r{bdim}(Q)$ to be realized by Boolean expressions
on nonempty families of linearizations.
Whether it would be best to allow $d=0$ in the definition
of $\r{bdim},$ making antichains an exception to~\eqref{d.bdim_prod},
or to keep the requirement $d>0$ in the definition
of $\r{bdim},$ is for those in the field to decide.)

Returning to the result $\r{bdim}(\2^6)=5,$
we remark that the $\!5\!$-tuple of orderings of $|\2^6|$ used in
showing that the value is $\leq 5,$ listed in a page-and-a-half-long
table at~\cite[pp.30-31]{5/6}, was apparently found by a
computer search~\cite[last paragraph of p.29]{5/6}, and
shows no evident pattern.
So we ask

\begin{question}\label{Q.5/6}
Are there a family of $5$ orderings of $|\2^6|,$
and a function $|\2^5|\to\{0,1\},$ yielding a Boolean
representation of $\2^6,$ which can be described conceptually,
and for which there is a conceptual proof that they
determine the standard ordering on that set?
\end{question}

(Though as noted above, the $5$ total orderings
of $|2^6|$ used in~\cite{5/6} show no evident pattern,
the description at~\cite[next-to-last paragraph of p.29]{5/6}
of the function $|\2^5|\to\{0,1\}$ is pleasantly simple:
it takes a $\!5\!$-tuple to $1$ if and only if that
$\!5\!$-tuple has at least $4$ terms equal to $1.)$

\section{Other sorts of dimension in the literature}\label{S.other}

A different variant of the dimension of a
poset $P,$ also described in~\cite{comparing}, is
the {\em local dimension,} $\r{ldim}(P),$ the least $d$ such that
there exists a family of linearizations of {\em subposets} of $P$
such that each member of $P$ appears in $\leq d$ of these linearized
subposets, and such that for every pair
$(x,y)$ of distinct elements of $P,$ there are
enough linearized subposets containing both so that the
relationship between $x$ and $y$ in $P$ (an order-relation, or
incomparability) is the relation between their images in the
product of these linearizations.
As with the Boolean dimension, this is $\leq\dim(P),$ and
considerably less for the $S_n$ (in this case always $\leq 3.$
Hint: use $2$ linearizations of the whole set $|S_n|,$
and one additional linearization of each of the $n$
pairs $\{a_i,b_i\}).$
The behavior of the local dimension on product posets
is studied in \cite[\S3]{KMMSSUW}.

Still another variant: \ If instead of looking at lists
of linearizations
of $<_P$ such that every relation $x\not<_P y$ is realized
in at least one member of our list, and letting the dimension of
$P$ be the least cardinality of such a list, one can look at
lists of linearizations with a ``weight'',
a positive real number, attached to each linearization, such that
$x\not<_P y$ if and only if the sum of the weights of the
linearizations for which $x>y$ is at least $1,$
and define the {\em fractional dimension} of $P$ to be the infimum,
over all weighted lists which determine the ordering of $P$
in this way, of the sum of the weights of the listed linearizations.
(The term ``weight'' is not used in the literature; it is my way
of giving an intuitive description of the definition.)
See~\cite[p.5]{KK+WT} and references given there.
\cite{KK+WT} also introduces several dimension-like
functions specific to the type of posets there called
convex geometries.

Less exotic invariants of posets considered in~\cite{comparing} and
elsewhere are
the {\em height,} i.e., the supremum of the cardinalities of chains
contained in $P,$ and the {\em width,} the supremum of the
cardinalities of antichains in $P.$
The function associating to a poset $P$
the largest $n$ such that $P$ contains a subposet isomorphic
to $S_n$ is denoted $\r{se}(P),$ and studied
in \cite{random} and \cite[\S5.2.1]{Trotter}.

The function $\dim(P),$ though usually (as in this note) simply called
the dimension of $P,$ is sometimes, as in the title of~\cite{comparing},
called the {\em Dushnik-Miller} dimension, to distinguish
it amid this plethora of variants.

\section{When an infinite poset has finite dimension}\label{S.inf_fr_fin}

I have left this topic to the end because it assumes
familiarity with a very different
technique from those used in the other sections.
Namely, by a straightforward application of the
{\em Compactness Theorem} of
first-order logic \cite[Theorem VI.2.1(b)]{E+F+T},
\cite[Theorem 6.1.1]{Hodges}, applied to a language with
a constant for each member of $|P|,$
and using version~(i) of Definition~\ref{D.dim} above,
one can immediately verify

\begin{proposition}\label{P.inf_fr_fin}
Let $P$ be a poset of arbitrary cardinality, and $d$ a positive integer.
Then $P$ has dimension $d$ if and only if $d$ is the
supremum of the dimensions of all finite subposets of~$P.$\qed
\end{proposition}

(I originally had two rather complicated
proofs of this result -- one using a carefully chosen
ultraproduct of the finite subposets of $P,$ the other using
a tricky transfinite recursion;
but Theodore Slaman pointed out that the Compactness Theorem
gives the result immediately.)

Immediate consequences regarding a couple of the other invariants
we have defined are

\begin{corollary}\label{C.bddim_fr_fin}
For any finite-dimensional poset $P,$ its
bounded dimension $\r{bd\mhyphen dim}(P)$
is equal to the supremum of the values of
$\r{bd\mhyphen dim}(P')$ on finite subposets $P'\subseteq P.$\qed
\end{corollary}

\begin{corollary}\label{C.abs_fr_fin}
For any finite-dimensional poset $P,$ its absorbency
$\r{abs}(P)$ is equal to the {\em infimum} of the values of
$\r{abs}(P')$ over the finite subposets $P'$ of $P$
satisfying $\dim(P')=\dim(P).$\qed
\end{corollary}

I wonder about
\begin{question}\label{Q.big_fr_smaller}
Do there exist {\em infinite} cardinals $\kappa$ and $\lambda$ such
that for every poset $P,$
if all subposets $P'\subseteq P$ of cardinality $<\kappa$ have
dimension $\leq\lambda,$ then $P$ itself has dimension $\leq\lambda$?
\end{question}

Proposition~\ref{P.inf_fr_fin} is the corresponding statement
with $\aleph_0$ in the role of $\kappa,$ and the natural number
$d$ in place of the infinite cardinal~$\lambda.$

\section{Acknowledgements}\label{S.ackn}

I am indebted to Jonathan Farley for pointing me to
\cite{S3x2x2}, which started me on this work,
to him and to William Trotter for pointers
to much of the literature cited below, to Richard Stanley
for catching a serious typo in an earlier version,
to Theodore Slaman for the simplified way of seeing
the truth of Proposition~\ref{P.inf_fr_fin}, to Csaba Biro
for informing me that the result proved in Lemma~\ref{L.leq-Bool}
is indeed known to workers in the field, and to several referees
for helpful comments.


\begin{thebibliography}{00}

\bibitem{KB} Kirby A. Baker,
{\em Dimension, join-independence, and breadth in partially ordered
sets}, Honors Thesis, Harvard University, 1961 (unpublished), 28pp.
\url{https://www.math.ucla.edu/~baker/res/po/undergraduate_thesis.pdf}

\bibitem{comparing} Fidel Barrera-Cruz, Thomas Prag, Heather C. Smith,
Libby Taylor, and William T. Trotter,
{\em Comparing Dushnik-Miller dimension, Boolean
dimension and local dimension}, Order {\bf 37} (2020) 243-269.
MR4123380

\bibitem{random} C. Bir\'o, P. Hamburger, H. A. Kierstead, A. P\'or,
W. T. Trotter, and R. Wang,
{\em Random bipartite posets and extremal problems},
Acta Math. Hungar. {\bf 161} (2020) 618-646.
MR4131937

\bibitem{5/6} Marcin Bria\'{n}ski, J\k{e}drzej Hodor,
Hoang La, Piotr Micek, and Katzper Michno,
{\em Boolean dimension of a Boolean lattice,}
Order {\bf 42} (2025) 25-36.
MR4918665

\bibitem{S3x2x2} Ilya Bogdanov, solution to
{\em Find an order-embedding of $S_3\times\2\times\2$
into $\mathbb{Z}^4$},
\url{https://mathoverflow.net/questions/454278/find-an-order-embedding-of-s-3-times-bf2-times-bf2-into-mathbb-z4},
(2023).
(This solution does better than the problem asks for, giving an
embedding into a product of $3$ chains.)

\bibitem{BD} Ben Dushnik,
{\em Concerning a certain set of arrangements},
Proc. Amer. Math. Soc. {\bf 1} (1950) 788-796.
MR0038922

\bibitem{D+M} Ben Dushnik and E. W. Miller,
{\em Partially ordered sets,}
Amer. J. Math. {\bf 63} (1941), 600-610.
MR0004862

\bibitem{E+F+T}
Heinz-Dieter Ebbinghaus, J\"{o}rg Flum and Wolfgang Thomas,
{\em Mathematical logic}, Third edition. Graduate Texts in Mathematics,
{\bf 291}.
Springer, 2021. ix+304 pp.
MR4273297

\bibitem{F+T+W} Stefan Felsner, Torsten M\"{u}tze and
Maximilian Wittmann,
{\em Order Dimension, Grids, and Products},
Order {\bf 42} (2025) 811-827.
MR4989463

\bibitem{ZF} Zolt\'{a}n F\"{u}redi,
{\em The order dimension of two levels of the Boolean lattice,}
Order {\bf 11} (1994) 15-28.  MR1296231

\bibitem{G+N+T} Giorgio Gambosi, Jaroslav Ne\v{s}et\v{r}il and Maurizio Talamo,
{\em On locally presented posets,}
Theoret. Comput. Sci. {\bf 70} (1990) 251-260.
MR1044542

\bibitem{Hodges}
Wilfrid Hodges, {\em Model theory},
Encyclopedia of Mathematics and its Applications, {\bf 42}.
Cambridge University Press, Cambridge, 1993.  xiv+772\,pp.
MR1221741

\bibitem{H+K+T}
G. H. Hurlbert, A. V. Kostochka, and L. A. Talysheva,
{\em The dimension of interior levels of the Boolean lattice},
Order {\bf 11} (1994) 29-40.
MR1296232

\bibitem{DK} David Kelly,
{\em On the dimension of partially ordered sets},
Discrete Math. {\bf 35} (1981) 135-156.
MR0620667

\bibitem{K+T} David Kelly and William T. Trotter, Jr.,
{\em Dimension theory for ordered sets}, pp.\,171-211 in
{\em Ordered sets \textup{(}Banff, Alta, 1981\textup{)}},
NATO Adv. Study Inst. Ser. C: Math. Phys. Sci.,
{\bf 83} (1982).
MR0661294

\bibitem{KMMSSUW} Jinha Kim, Ryan R. Martin, Tomáš Masařík, Warren Shull, Heather C.  Smith, Andrew Uzzell, Zhiyu Wang,
{\em On difference graphs and the local dimension of posets,}
European J. Combin. {\bf 86} (2020) 103074, 13 pp.
MR4046479

\bibitem{KK+WT} Kolja Knauer and William T. Trotter,
{\em Concepts of dimension for convex geometries},
SIAM J. Discrete Math. {\bf 38} (2024) 1566-1585.
MR4748797

\bibitem{Ko+Ta} A. V. Kostochka and L. A. Talysheva,
{\em The dimension of interior-levels of the Boolean lattice. II},
Order {\bf 15} (1998) 377-383
(2000).
MR1741111

\bibitem{lin} Chiang Lin,
{\em The dimension of the cartesian product of posets},
Discrete Math. {\bf 88} (1991) 79-92.
MR1099269

\bibitem{indep_of_n} MathOverflow member Tri,
{\em Dimension of the cartesian product of a poset and a chain},
\url{https://mathoverflow.net/questions/401495/dimension-of-the-cartesian-product-of-a-poset-and-a-chain}\,,
(2021).

\bibitem{boolean} J. Ne\v{s}et\v{r}il, P. Pudl\'{a}k,
{\em A Note on Boolean Dimension of Posets}, in {\em Irregularities of
partitions}, Algorithms Combin. {\bf 8} (1989) 137-140.
MR0999938

\bibitem{Ore} Oystein Ore,
{\em Theory of graphs},
American Mathematical Society Colloquium Publications,
vol. XXXVIII.  American Mathematical Society, Providence, RI
1962. x+270~pp.
MR0150753

\bibitem{KR} Klaus Reuter,
{\em On the dimension of the Cartesian product of relations and orders},
Order {\bf 6} (1989) 277-293.
MR1048098

\bibitem{KR90} Klaus Reuter,
{\em On the order dimension of convex polytopes},
European J. Combin., {\bf 11} (1990) 57-63.
MR1034145

\bibitem{Sz} Edward Szpilrajn,
{\em Sur l'extension de l'ordre partiel,}
Fundamenta Mathematicae, {\bf 16} (1930) 386-389.
\url{http://matwbn.icm.edu.pl/ksiazki/fm/fm16/fm16125.pdf}

\bibitem{Trotter89} William T. Trotter, Jr.,
{\em The dimension of the cartesian product of partial orders},
Discrete Math. {\bf 53} (1985) 255-263.
MR0786494

\bibitem{Trotter} William T. Trotter,
{\em Dimension for posets and chromatic number for graphs}, pp.73-96 in
{\em 50 Years of Combinatorics, Graph Theory, and Computing},
Discrete Math. Appl.,
CRC Press, Boca Raton, FL, (2020).
MR4368165

\bibitem{T+W} W.T. Trotter and B. Walczak,
{\em Boolean dimension and local dimension},
Elec. Notes Discret. Math. {\bf 61} (2017) 1047-1053.
\url{https://trotter.math.gatech.edu/papers/142-Boolean_dimension_and_local_dimension.pdf}\,.
zbMATH:1397.06003.

\end{thebibliography}
\end{document}